\documentclass[11pt]{amsart}
\usepackage{fullpage}
\usepackage{latexsym}
\usepackage{amssymb}
\usepackage{pb-diagram}
\usepackage{amsmath}
\usepackage{amsthm}
\usepackage{mathrsfs}

\newtheorem{theorem}{Theorem}[section]
\newtheorem{definition}[theorem]{Definition}
\newtheorem{prop}[theorem]{Proposition}
\newtheorem{lemma}[theorem]{Lemma}
\newtheorem{remark}[theorem]{Remark}

\newtheorem{ex}[theorem]{Example}

\newtheorem{Cor}[theorem]{Corollary}

\DeclareMathOperator{\Ima}{Im}
\DeclareMathOperator{\image}{im}
\DeclareMathOperator{\diag}{diag}

\DeclareMathOperator{\Is}{Is}

\DeclareMathOperator{\Pic}{Pic}
\DeclareMathOperator{\id}{id}
\DeclareMathOperator{\Endo}{End}
\DeclareMathOperator{\Sp}{Sp}

\DeclareMathOperator{\Aut}{Aut}

\DeclareMathOperator{\Nm}{Nm}
\def\CC{\mathbb{C}}
\def\NN{\mathbb{N}}
\def\PP{\mathbb{P}}
\def\QQ{\mathbb{Q}}

\def\ZZ{\mathbb{Z}}
\def\cA{\mathcal{A}}
\def\cJ{\mathcal{J}}

\def\cO{\mathcal{O}}

\def\To{\longrightarrow}

\def\SH{\mathfrak{h}}

\def\Mapsto{\mapstochar\longrightarrow}
\title{Non-simple principally polarised abelian varieties}
\author{Pawe\l{} Bor\'owka}
\address{Institute of Mathematics, Jagiellonian University, Krakow}
\email{Pawel.Borowka@uj.edu.pl} 
\begin{document}
% *************** Main matter ***************
%\mainmatter
%\input{ack/ack.tex}

\begin{abstract}
The paper investigates the locus of non-simple principally polarised abelian $g$-folds. We show that the
irreducible components of this locus are $\Is^g_{D}$,
defined as the locus of principally polarised $g$-folds having an
abelian subvariety with induced polarisation of type
$D=(d_1,\ldots,d_k)$, where $k\leq\frac{g}{2}$. The main theorem produces Humbert-like equations for irreducible components of $\Is^g_{D}$ for any $g$ and $D$.
Moreover,
there are theorems which characterise the Jacobians of curves that
are \'etale double covers or double covers branched in two or four points.
\end{abstract}
\maketitle

\section{Introduction}
A common approach to understand the geometry of the moduli space of abelian varieties
is to use ideas coming from the geometry of curves. That is possible because of the Torelli theorem, which says that the Jacobian completely
characterises the curve. Because of that, many geometric constructions from the theory of curves give rise to
interesting constructions in the theory of Jacobians. 
One remarkable construction is the Prym construction,
which gives a subvariety of a Jacobian for any finite cover of curves.
More precisely, every cover of curves $f:C\longrightarrow C'$ induces a pullback map $f^*:JC'\longrightarrow JC$. Therefore $JC$ is a non-simple abelian variety, as it contains $\image f^*$ and the complementary abelian subvariety,  
called the Prym variety of the cover.

The motivation behind results of this paper is to understand the locus of non-simple abelian varieties itself. One can ask:
\begin{center}
$(\star)$ \textit{What does the locus of non-simple principally polarised abelian $g$-folds look like?} 
\end{center}

For abelian surfaces, a non-simple abelian surface contains an elliptic curve, and Humbert \cite{H} proved that
the locus of non-simple principally polarised abelian surfaces is the union of countably many irreducible
surfaces called Humbert surfaces, in the moduli space. The Humbert surfaces are indexed by the degree of the polarisation restricted to an elliptic curve.

In Section 3, we propose a definition of the generalised Humbert locus, denoted by $\Is^g_{D}$,
which is the locus of principally polarised abelian $g$-folds having an abelian subvariety with induced polarisation of type
$D=(d_1,\ldots,d_k)$, where $k\leq\frac{g}{2}$.
The definition was stated by O. Debarre in \cite[p 259]{Deb}, denoted by $\mathscr{A}^{\delta}_{g',g-g'}$. Then he proves irreducibility of $\Is^g_{D}$, using irreducibility of some moduli space. 
The main result of Section 3 is Proposition \ref{irrofis} that also states that $\Is^g_{D}$ is irreducible. Both ideas of proofs are similar, but the proof presented in this paper is explicit.

Using the fact that every non-simple abelian $g$-fold belongs to some $\Is^g_{D}$, we get that
\textit{the only discrete invariants of the locus of non-simple  principally polarised abelian varieties
are the dimensions of subvarieties and the type of the induced polarisation on the smaller one.}
Moreover, all possibilities of $k\leq\frac{g}{2}$ and $D=(d_1,\ldots,d_k)$ can occur.

The main results of this paper are contained in Section 4, where we find equations of preimage of $\Is^g_{D}$ in the Siegel space $\SH_g$, for all $g>2$.
Because the codimension is bigger than $1$, we could not find symplectic invariants similar to Humbert's discriminant, so Theorem \ref{thmIsgd}
 provides only a set of linear equations of one particular irreducible component of the
preimage in $\SH_g$ of $\Is^g_D$. 

The main result of the paper states that \textit{a principally polarised abelian $g$-fold $(A,H)$ is non-simple, thus containing an abelian subvariety $M\subset A$ of dimension $k\leq\frac{g}{2}$ such that  $H|_M$ is of type $D=(d_1,\ldots,d_k)$, if and only if one can find a period matrix
$Z_A=[z_{ij}]\in \SH_g$ 
satisfying the linear equations
\begin{align*}
&z_{ij}=d_iz_{(g+1-i)j},&\quad &i=1,\ldots,k, j=1,\ldots,k&\\ 
&z_{ij}=0,&\quad &i=k+1,\ldots,g-k, j=1,\ldots,k&
\end{align*}
such that $\Lambda_A=\left<Z_A\ \id_g\right>$ and $A\cong\CC^g/\Lambda_A$.}

Moreover, using results from Section \ref{sec4} one can easily find  many other sets of equations.

The next step of the investigation is to \textit{understand the locus of Jacobians that are non-simple abelian varieties.} In full generality, due to the Schottky problem, this task is hard. However, in some specific situations, we have a complete answer.
When the genus of the curve is 2, the answer can be easily extracted from the work of Humbert \cite{H}:
the Jacobian of a genus 2 curve $C$ is non-simple and contains an elliptic curve $E$ with the induced polarisation of type $n$ if and only if the curve $C$ is an $n:1$ cover of $E$. 

In genus 3, the answer is well known and completely analogous. Proposition \ref{propIs3n} states that
the Jacobian of a genus 3 curve $C$ is non-simple and contains an elliptic curve $E$ 
with the induced polarisation of type $n$ if and only if the curve $C$ is an $n:1$ cover of $E$. 
However, \textit{if we restrict our attention to hyperelliptic Jacobians that belong to $\Is^3_2$ then
we find a nice characterisation of \'etale double covers of genus 2 curves:} see Proposition \ref{propIs32} and Corollary \ref{remIs32}.

Proposition \ref{propIs32} is generalised in Theorem \ref{genis32thm} which, 
roughly speaking, says that
if the Jacobian of a curve
contains an abelian subvariety of half the dimension and the type of the induced polarisation is twice the principal polarisation, 
then there is a double cover of curves that yields the Jacobian and the subvariety. 

Some results of this paper are contained in the PhD thesis of the author. The author would like to thank his supervisor Gregory Sankaran, University of Bath and Jagiellonian University in Krak\'ow for all support and Olivier Debarre for pointing out results from his paper.

\section{Preliminaries}
In this section we review the well-known facts and set up notation. For more details, see \cite{LB}.

By an abelian variety, we mean a projective complex torus. An abelian variety $A$ is isomorphic to $\CC^g/\Lambda_A$, where $\Lambda_A$ is a lattice of maximal rank.
For a $g\times 2g$ matrix $[Z\ D]$, the lattice spanned by the column vectors will be denoted by $\langle Z\ D\rangle$. We can always write $\Lambda_A$ in the form
$$\Lambda_A=\left<Z\ D\right>,$$
in such a way that $Z$ belongs to the Siegel space
$$\SH_g=\{Z=[z_{ij}]\in M(g,\CC):Z=\ ^tZ,\ \Ima Z>0\}$$ and a diagonal matrix $D=\diag(d_1,\ldots,d_g)$ with positive integer values $d_i$ such that $d_i|d_{i+1}, i=1,\ldots,g-1$. Moreover, on such a variety we always choose a polarisation of type $D$, usually denoted by $H$.

For any
 polarisation $H$ we define an isogeny $\phi_H\colon A\To \hat{A}$  with analytic representation given by $(\Ima Z)^{-1}$. The
exponent of $\phi_H$ is denoted by $e(H)$ and called the exponent of
the polarisation. By $K(H)$ we denote the kernel of $\phi_H$. Using a decomposition for $H$ one proves that $K(H)$ is isomorphic to $(\ZZ_{d_1}\times\ldots\times\ZZ_{d_g})^2$. On $K(H)$ there exists a multiplicative alternating nondegenerate form 
$$e^H(w_1,w_2)=
\exp(-2\pi i (\Ima H)(v_1,v_2))\in\CC^*,
$$
where $v_1,v_2$ are any preimages of $w_1,w_2$ in $\CC^g$.

We denote by $\cA_g=\SH_g/\Sp_{2g}(\ZZ)$ the moduli space of principally polarised abelian $g$-folds and by $\cA_D$ the moduli space of $D$-polarised abelian varieties. Inside $\cA_g$ we denote the locus of Jacobians by $\mathcal{J}$ and the locus of hyperelliptic Jacobians by $\mathcal{JH}$.

\subsection{Complementary abelian subvarieties}
In Section~\ref{Hum}, we will try to understand the locus of
non-simple abelian varieties. The idea is to improve the statement of uniqueness in Poincar\'e's Complete Reducibility Theorem.
Therefore we need to recall the following definitions and Poincar\'e's Reducibility Theorems. For details we refer to \cite{LB}.
\begin{definition}\cite[p.132]{LB}
Let $\iota\colon M\To A$ be an abelian subvariety of a principally
polarised abelian variety $(A,H)$. Then $\iota^*H$ is a polarisation
on $M$, denoted also by $H|_M$. Define the \emph{exponent of $M$} by $e(M)=e(\iota^*H)$. Moreover, we
define the \emph{norm-endomorphism of $A$ associated to $M$} by
$$
\Nm_M=\iota (e(M)\phi_{\iota^*H}^{-1})\hat{\iota}\phi_H.
$$
$\varepsilon_M=\frac{1}{e(M)}\Nm_M\in\Endo_\QQ(A)$ is called
the \emph{associated symmetric idempotent}.

Conversely for any symmetric idempotent
$\varepsilon\in\Endo_\QQ(A)$ there exists $n\in\NN$
such that $n\varepsilon\in\Endo(A)$, and we can define the abelian
subvariety $A^\varepsilon=\image(n\varepsilon)$.
\end{definition}

The next theorem is the main tool in proving Poincar\'e's Reducibility
Theorems.
\begin{theorem}\cite[Thm 5.3.2]{LB}
The assignments $M\mapsto\varepsilon_M$ and $\varepsilon\mapsto
A^\varepsilon$ are inverse to each other and give a bijection between
the sets of abelian subvarieties of $A$ and symmetric idempotents in
$\Endo_\QQ(A)$ 
\end{theorem}
%\qed

The main advantage of translating the existence of subvarieties into
symmetric idempotents is that the latter have an obvious
canonical involution $\varepsilon\mapsto 1-\varepsilon$. This leads to
the following definition.

\begin{definition}\cite[p.125]{LB}\label{comabsub}
Let $A$ be a polarised abelian variety. Then the polarisation induces
a canonical involution on the set of abelian subvarieties of $A$:
$$M\mapsto N=A^{1-\varepsilon_M}$$ We call $N$ the \emph{complementary
abelian subvariety} of $M$ in $A$, and $(M,N)$ a \emph{pair of complementary
abelian subvarieties}.
\end{definition}

In this paper we often consider products of principally polarised abelian
varieties, and therefore we introduce the following notation. 
\begin{definition}
Suppose $k$ and $g$ are integers with
$0<k\leq \frac{g}{2}$, and $D=(d_1,\ldots, d_k)$ is a polarisation type. 
The \emph{complementary polarisation type}, denoted 
$\tilde{D}$, is the $(g-k)$-tuple $(1,\ldots,1,d_1,\ldots, d_k)$.
\end{definition}
If $(M,H_M)$ and $(N,H_N)$ are polarised abelian varieties
of types $D$ and $\tilde{D}$ then, even if not
written explicitly, we will treat the product $M\times N$ as the
$(D,\tilde{D})$-polarised variety with the canonical product
polarisation. Strictly speaking $(D,\tilde{D})$ is not a polarisation type, so one needs to permute the coordinates.
The following proposition shows that indeed the complementary abelian subvariety in the principally polarised abelian 
variety has a complementary polarisation type

\begin{prop}\label{cmplandroots}
Let $(A,H)$ be a principally polarised abelian variety. The
following conditions are equivalent:
\begin{enumerate}
\item there exists $M\subset A$ such that $H|_M$ is of type
  $D$.
\item there exists $N\subset A$ such that $H|_N$ is of type
  $\tilde{D}$.
\item there exists a pair $(M,N)$ of complementary abelian subvarieties
  in $A$ of types $D$ and $\tilde{D}$ respectively.
\item there exists a polarised isogeny $\rho\colon M\times N\To A$ with $$\ker\rho\cap (M\times\{0\})=\ker\rho\cap(\{0\}\times N)=\{0\}.$$

\end{enumerate}
\end{prop}
\begin{proof}
The equivalence of conditions (1), (2), (3) follows from the
definition and~\cite[Cor 12.1.5]{LB}.

$(3)\Rightarrow(4)$ is a consequence of~\cite[Cor
  5.3.6]{LB}. The condition on the kernel states that
$\rho|_{M\times\{0\}}$ and $\rho|_{\{0\}\times N}$ are
inclusions. 

$(4)\Rightarrow(3)$ Let us denote the inclusions by
$\iota_M=\rho|_{M\times\{0\}}$ and $\iota_N=\rho|_{\{0\}\times N}$
Then $\rho(m,n)=\iota_M(m)+\iota_N(n)$ and so
$$\varepsilon_M+\varepsilon_N=
\iota_M\phi^{-1}_{\iota_M^*H}\hat{\iota_M}\phi_H+
\iota_N\phi^{-1}_{\iota_N^*H}\hat{\iota_N}\phi_H=
\rho\phi^{-1}_{\rho^*H}\hat{\rho}\phi_H=\phi^{-1}_H\phi_H=1 $$
\end{proof}

\begin{theorem}\cite[Thm 5.3.5]{LB}\label{prthm}
Let $(A,H)$ be a polarised abelian variety and $(M,N)$ a pair of
complementary abelian subvarieties of $A$. Then the map
$$
(\Nm_M,\Nm_N)\colon A\To M\times N
$$ 
is an isogeny.
\end{theorem}

\begin{theorem}\cite[Thm 5.3.7]{LB}\label{pcrthm}
For an abelian variety $A$ there is an isogeny
$$
A\To A_1^{n_1}\times\ldots\times A_r^{n_r}
$$ 
with simple
abelian varieties $A_i$ not isogenous to each other. Moreover the
abelian varieties $A_i$ and integers $n_i$ are uniquely determined up
to isogenies and permutations.
\end{theorem}

Theorems \ref{prthm} and \ref{pcrthm} are Poincar\'e's Reducibility and Complete Reducibility Theorems.

\subsection{Symplectic forms on finite abelian groups}\label{ssanty}

Later, we shall be interested in isotropic subspaces of $K(H)$ and their
behaviour under the action of the symplectic group. Therefore we need some basic facts
from the theory of finite symplectic
$\ZZ$-modules. We will use the fundamental theorem of finite
abelian groups.
\begin{theorem}
Any finite $\ZZ$-module $X$ can be written uniquely in the form
$\ZZ_{d_1}\times\ldots\times\ZZ_{d_k}$, where $d_1>1$
and $d_i|d_{i+1},\ i=1,\ldots,k-1$, for a unique $(d_1,\ldots,d_{k})$. 
We will call $(d_1,\ldots,d_{k})$  the \emph{type of the $\ZZ$-module}. 
\end{theorem}

\begin{definition}
We say that $X$ is of \emph{dimension $k$}
and a \emph{basis of $X$} is an image of any basis of $\ZZ^{k}$ by an
epimorphism 
$$
\ZZ^{k} \To \ZZ_{d_1}\times\ldots\times\ZZ_{d_k}.
$$

Assume $X$ is of the form
$(\ZZ_{d_1}\times\ldots\times\ZZ_{d_k})^2$. A \emph{symplectic
form on $X$} is the image of a symplectic form on
$\ZZ^{k}\times\ZZ^{k}$. 
\end{definition}

\begin{ex}\cite[Lem 3.1.4]{LB}
Let $H$ be a polarisation of type 
$(d_1,\ldots,d_{k})$, with $d_{m}=1, d_{m+1}>1$, for some $m$ on an abelian variety $A$ and consider a decomposition $V_1\oplus V_2$ for $H$. Then
$$K(H)=K(H)_1\oplus K(H)_2,\ \text{ with }\ K(H)_1\cong K(H)_2\cong \ZZ_{d_{m+1}}\times\ldots\times\ZZ_{d_k},$$
so both $K(H)_1$ and $K(H)_2$ are of the same type 
and of dimension $k-m$.
Moreover $K(H)$ has symplectic form $e^H$.
\end{ex}

It is convenient to work under the assumption that the domain and codomain are of the same
type. Therefore, we define
\begin{definition}
Let $(X,\omega_X)$ and $(Y,\omega_Y)$ be symplectic
$\ZZ$-modules of the same type. Then a $\ZZ$-linear map
$f\colon X\To Y$ is called an \emph{antisymplectic map} if for all
$x,\,y\in X$, we have 
\begin{equation}\label{eqanty}
\omega_X(x,y)=-\omega_Y(f(x),f(y)).
\end{equation}
\end{definition}
\begin{prop}\label{maxisosub}
Every antisymplectic map is a bijection and the inverse map
is also antisymplectic. Moreover, the space of
antisymplectic maps is modelled on $\Sp(X,\ZZ)$, i.e.\ for every
antisymplectic $f,\,g\colon X\To Y$, we have $g^{-1}\circ f\in
\Sp(X,\ZZ)$ and for all $s_X\in \Sp(X,\ZZ)$, the maps
$f\circ s_X$ are antisymplectic. 
By symmetry it is also modelled on $\Sp(Y,\ZZ)$.  

\end{prop}
\begin{proof}
By equation \eqref{eqanty}, if $f(x)=0$, then $\omega_X(x,y)=0$ for
every $y$, so $x=0$, which means $f$ is injective. Bijectivity comes
from the domain and codomain having the same type. The rest of the proposition comes
from the fact that $(-1)\cdot(-1)=1$. 
\end{proof}

\begin{prop}\label{maxisosub2}
Let $(X,\omega_X)$ and $(Y,\omega_Y)$ be symplectic
$\ZZ$-modules of the same type. Consider $(X\oplus Y,\omega_X+\omega_Y)$.
Then the set of graphs of antisymplectic maps is the set of maximal
isotropic subspaces of $X\oplus Y$ intersecting $X$ and $Y$ only in
$\{0\}$. 

In particular, all maximal
isotropic subspaces of $X\oplus Y$ intersecting $X$ and $Y$ only in
$\{0\}$ are equivalent under the actions of symplectic groups on $X$ and on $Y$.
\end{prop}
\begin{proof}
It is obvious that the graph of an antisymplectic map is an isotropic
subspace with the desired properties. For the converse, let $Z$ be a
maximal isotropic subspace. 
Then the projections $\pi_X\colon Z\To X$ and
$\pi_Y\colon Z\To Y$ are bijections. Moreover $\pi_Y\circ\pi_X^{-1}$ is
antisymplectic and $Z$ is the graph of $\pi_Y\circ\pi_X^{-1}$.

The second part of the proposition is a direct application of Proposition \ref{maxisosub}.
\end{proof}

\section{Generalised Humbert locus}\label{Hum}

\subsection{Background -- Humbert surfaces}
We begin the study of moduli of non-simple
abelian varieties with the surface case, by recalling the Humbert surfaces of square discriminant.

\begin{theorem}\label{humb}
Let $p$ be a positive integer. Let $(A,H)$ be a principally polarised abelian
surface.  The following conditions are equivalent:
\begin{enumerate}
\item there exists an elliptic curve $E\subset A$ such that
  $H|_E$ is of type $p$;
\item there exists an exact sequence 
$$
0\To E \To A \To F \To 0,
$$ 
and therefore its dual
$$
0\To F \To A \To E \To 0,
$$ 
such that the induced map $E\times F\To A$ is an isogeny of degree $p^2$;
\item there exists a pair $(E,F)$ of complementary elliptic curves in
  $A$ of type $p$;
\item $\Endo(A)$ contains a primitive symmetric endomorphism $f$ with
  discriminant $p^2$;
\item $\Endo(A)$ contains a symmetric endomorphism $f$ with analytic
  and rational representations given by
$$
\left[\begin{array}{cc}
0&0\\
-1&p\\
\end{array}\right],\quad\left[\begin{array}{cccc}
0&-1&0&0\\
0&p&0&0\\
0&0&0&0\\
0&0&-1&p\\
\end{array}\right];
$$
\item $(A,H)$ is isomorphic to an abelian surface defined by a period matrix
$$
\left[
\begin{array}{cccc}
pt_2&t_2&1&0\\
t_2&t_3&0&1\\
\end{array}\right],
$$ with elliptic curves defined by period matrices $[t_2\ 1]$
and $[pt_3-t_2\ 1]$ embedded as $s\mapsto (ps,s)$ and
$s\mapsto (0,s)$.
\end{enumerate}
\end{theorem}
\begin{proof}
We have already proved the equivalence of (1), (2) and~(3) in
Proposition~\ref{cmplandroots} in a more general setting.
The equivalence of (4), (5) and (6) is a direct application of \cite[Section 4]{BW}.
Next, we will show that $(3)\implies(4)$ and $(6)\implies(1)$.

For the first implication, take $f$ to be the norm-endomorphism
associated to either elliptic curve. By \cite[Thm 5.3.4]{LB}, $f$ is
primitive, symmetric and has characteristic polynomial $f^2-pf$. So
its discriminant equals $p^2$.

For the second implication, to simplify notation, we write $t'=\Ima(t)$
for any $t\in\CC$. Then 
$$
\det(\Ima Z)=t_2'(pt_3'-t_2'),
$$
$$
H=(\Ima Z)^{-1}=\det (\Ima Z)^{-1}
\left[\begin{array}{cc}
t_3'&-t_2'\\
-t_2'&pt_2'\\
\end{array}\right].
$$
Define $\iota_E\colon s\mapsto (ps,s)$. Its analytic
representation is given by the matrix $\left[\begin{array}{c}
    p\\1\\
\end{array}\right].
$
Then $\phi_{\iota_E^*H}=\hat{\iota}_E\circ\phi_H\circ\iota_E$ is defined by 
\begin{equation*}
\left[\begin{array}{cc}
p&1 \\
\end{array}\right] 
\det (\Ima Z)^{-1}
\left[\begin{array}{cc}
t_3'&-t_2'\\
-t_2'&pt_2'\\
\end{array}\right]
\left[\begin{array}{c}
p\\
1\\
\end{array}\right] 
=\left[\begin{array}{cc}
p&1 \\
\end{array}\right] 
\left[\begin{array}{cc}
\frac{pt_3-t_2'}{t_2'(pt_3-t_2')}\\
0\\
\end{array}\right]
=p
\left[\begin{array}{c}
(t_2')^{-1}\\
\end{array}\right],
\end{equation*}
so $H|_E$ is of type $p$.
\end{proof}

Condition~(6) of Theorem~\ref{humb} implies that in $\cA_{2}$,
the locus of all principally polarised abelian surfaces satisfying the
above conditions is the image of the surface given by the equation
$t_1=pt_2$ in $\SH_2$, and therefore it
is an irreducible surface in $\cA_{2}$.

\begin{definition}
The locus in $\cA_{2}$ of all principally polarised abelian
surfaces that satisfy the conditions of Theorem \ref{humb} is called the Humbert
surface of discriminant $p^{2}$.
\end{definition}

Humbert showed more in~\cite{H}. He found the equations defining the
preimage in $\SH_2$ of all Humbert surfaces.  To be precise,
any 5-tuple of integers without common divisor $(a,b,c,d,e)$ with the
same discriminant $\Delta = b^2 - 4ac - 4de$ gives us the so-called
singular relation
$$at_1 + bt_2 + ct_3 + d(t_2^2 - t_1t_3) + e = 0$$
In other words, the period matrix $Z=\left[\begin{array}{cc}
t_1&t_2\\
t_2&t_3\\
\end{array}\right]\in\SH_2$ is a solution to a singular
relation with $\Delta=p^2$ if and only if the abelian surface
$A_Z=\CC^2/(Z\ZZ^2+\ZZ^2)$ contains an elliptic
curve with induced polarisation of type $p$. 

If we recall that $\cA_2=\SH_2/\Sp(4,\ZZ)$, then it means that all matrices 
which satisfy the singular equation for some $\Delta=p^2$ form a symplectic orbit.
Then condition 6 of Theorem \ref{humb} says that there always exists a normalised period matrix i.e.
such that $$a=-1,\  b=p,\  c=d=e=0.$$

\subsection{Generalised Humbert locus}

We would like to generalise the notion of Humbert surface to higher dimensions.  There
are a few immediate problems that arise. Firstly, Humbert surfaces are
divisors in $\cA_{2}$ globally defined by one equation in
$\SH_2$, whereas in higher dimensions that is not the
case. 
Secondly, all elliptic curves are essentially canonically principally
polarised whereas in higher dimensions polarisations are much richer.

If we consider an abelian subvariety of an abelian variety, then the two obvious discrete invariants are the dimension of the subvariety and the type of the induced polarisation. So we define
\begin{definition}
The \emph{generalised Humbert locus of type $D=(d_1,\ldots,d_k)$
in dimension $g$}, denoted by $\Is^g_{D}$, is the locus in
$\cA_g$ of principally polarised $g$-folds $X$ such that
there exists a $k$-dimensional subvariety $Z$ of $X$ such that 
the restriction of the polarisation from $X$ to $Z$ is of type
$D$. If $d_1=d_k$ then we say it is of principal type.
\end{definition}
\begin{remark}
The same definition was proposed by O.Debarre in \cite[p 259]{Deb}, denoted by $\mathscr{A}^{\delta}_{g',g-g'}$.
\end{remark}
Firstly some obvious remarks and connections with previously known notions:

\begin{enumerate}
\item Every non-simple principally polarised abelian variety belongs
  to a generalised Humbert locus for some $g$ and $D$.
\item The name comes from the fact that for surfaces, it gives back
  Humbert surfaces of discriminant $D^2$. The word $\Is$ is an abbreviation of (polarised) isogenous to a product.

\item Note that principal type does not mean $d_k=1$. If $d_k=1$ then
  the isogeny from Proposition~\ref{cmplandroots} is actually an
  isomorphism and we get only the locus of products of principally
  polarised abelian varieties.  If we restrict the generalised Humbert
  locus of principal type to Jacobians of smooth curves then we get
  Jacobians containing Prym-Tyurin varieties.

\end{enumerate}

The first example of generalised Humbert locus arises in dimension three.
\begin{prop}\label{propIs3n}
A variety $A\in\Is^3_{n}$ is either a product of an elliptic curve with an abelian surface or the Jacobian of a smooth genus
$3$ curve which is an $n:1$ cover of an elliptic curve branched in
$4$ points and all such Jacobians are contained in $\Is^3_{n}$. In
other words, the only non-simple Jacobians are Jacobians of covers of
elliptic curves.
\end{prop}
\begin{proof}
By \cite[Cor 11.8.2]{LB} every  principally polarised abelian threefold is either a product or a Jacobian so we restrict our attention to Jacobians. If $JC\in\Is_n^3$,
then it contains an elliptic curve, say $E$. Taking the Abel-Jacobi map
composed with the dual of the inclusion, we get a map
$C\To E$. As $H|_E$ is of type $n$, it is an $n:1$ cover. Using the Hurwitz formula, we get that it has to be
branched in 4 points. Conversely, if we have an $n:1$ cover $\pi:C\To E$, 
then we can have the norm map $\Nm_\pi:JC\To JE=E$, given by
$\Nm_\pi(P-Q)=\pi(P)-\pi(Q)$.
Moreover $\Nm_\pi(\pi^{-1}(P'-Q'))=n(P'-Q')$ is a multiplication by $n$, so the induced polarisation is of type $n$.
\end{proof}

Before stating Proposition~\ref{propIs32}, Theorem \ref{genis32thm} and Corollary \ref{corIsJempty}, I
would like to note that the results are based on well-known ideas of Prym construction (see \cite{Mumfordprym} and \cite{LB}) and the generalised Torelli Theorem (see \cite{W}), but I was not able to find any exact
references in the literature. 

The first result is a characterisation of a $3$-dimensional family of
Jacobians of \'etale double covers of genus $2$ curves.

\begin{prop}\label{propIs32}
The locus of Jacobians of \'etale double covers of genus $2$ curves is $\Is^3_{2}\cap\mathcal{JH}$.
\end{prop}
\begin{proof}

Let $f\colon C\To C'$ be an \'etale double cover. It is defined
by a $2$-torsion point in $JC'$, say $\eta$. Then
$\ker(f^*)=\{0,\eta\}$ (\cite[Ex. B.14]{ACGH}). Therefore $f^*JC'=JC'/\langle\eta\rangle$ is a
$(1,2)$-polarised abelian surface, which is an abelian subvariety of
$JC$. Hence $JC\in\Is^3_2$. To finish the implication, let us note
that $C'$, being of genus $2$, has to be hyperelliptic and any \'etale
double cover of a hyperelliptic curve is hyperelliptic. This
implication can be easily deduced from the proof of part (a) of
\cite[Thm 7.1]{Mumfordprym}, or from~\cite{O}.

As for the other implication, let
$JC\in\Is^3_{2}\cap\mathcal{JH}$. Denote by $E$ an elliptic curve in
$JC$. Denote by $\iota_E$ the involution of $C$ which defines the
double cover and by $i_E$ its extension to $JC$. From construction,
$\image(1-i_E)=E$, and therefore $\epsilon_E=\frac{1-i_E}{2}$. On the
other hand, if we denote by $\iota$ the hyperelliptic involution on
$C$, then its extension to $JC$ is $(-1)$. This is because
for a branch point $Q$, we have $(P-Q)+(\iota(P)-Q)=0$, being the
principal divisor of a pullback of a meromorphic function on
$\PP^1$.  Now, $\iota\circ\iota_E$ is an automorphism on $C$
and its extension is $-i_E$.  Denoting by $Z=\image(1-(-i_E))$ and
$\epsilon_Z=\frac{1+i_E}{2}$, we immediately get that
$\epsilon_Z+\epsilon_E=1$ and so $(E,Z)$ is a pair of complementary
abelian subvarieties of $JC$.

Denote by $\iota_Z=\iota\circ\iota_E$ and let $C'=C/\iota_Z$ be the
quotient curve with the cover $f\colon C\To C'$ given by $P\mapsto
\{P,\iota_Z(P)\}$. Then $f^*(\{P,\iota_Z(P)\})=P+i_E(P)$, so
$Z=\image(f^*)$.

It is obvious that $\dim JC'=\dim Z=2$, so $C'$ is of genus $2$ and by
the Hurwitz Formula $f$ has to be an \'etale double cover.
\end{proof}
\begin{Cor}\label{remIs32}
Proposition~\ref{propIs32} says that a genus~$3$ curve is an \'etale double
cover of a genus~$2$ curve if and only if it is both hyperelliptic and
a double cover of an elliptic curve branched in~$4$ points.
\end{Cor}
\begin{proof}
Immediate from Proposition \ref{propIs32}
\end{proof}

Now, we would like to recall the tools which we use in the proof of Theorem \ref{genis32thm}.
\begin{theorem}[Generalised Torelli Theorem \cite{W}]\label{gentorthm}
Let $C$ be a smooth curve. Then
$$\Aut(C)=\begin{cases}\Aut(JC),\ \ \ \ \ \ \ \ \text{ if $C$ is hyperelliptic}\\
\Aut(JC)/(-1),\  \text{ if $C$ is not hyperelliptic}\end{cases}.$$
\end{theorem}

\begin{theorem}\cite[Prop 11.4.3 and Lem 12.3.1]{LB}
Let $f\colon C\To C'$ be a double cover of smooth curves.

If $f$ is \'etale then it is defined by a $2$-torsion point on $JC'$, say $\eta$, and $f^*$ factorises through $JC'/\eta$ which is embedded in $JC$. In this case the induced polarisation from $JC$ to $JC'/\eta$ is of type $(1,2,\ldots,2)$.

If $f$ is not \'etale then $f^*$ is injective and the induced polarisation from $JC$ to $JC'$ is twice the principal polarisation on $JC'$.
\end{theorem}

Denote by $\mathcal{D}^g_{b}$ the locus of curves of genus $g$ which are double covers branched in $b$ points, and by $\mathcal{JD}^g_{b}$ the locus of their Jacobians.
Let $\mathcal{HD}^g_{b}$ be the locus of hyperelliptic curves of genus $g$ which are double covers branched in $b$ points, and $\mathcal{JHD}^g_{b}$ their Jacobians.
The following theorem is the generalisation of Proposition \ref{propIs32}.
\begin{theorem}\label{genis32thm}
Let $\underline{2}=(2,\ldots,2)$ be a $g$-tuple of $2$'s. Then for odd dimension we have:
$$\Is^{2g+1}_{\underline{2}}\cap\mathcal{J}=\mathcal{JD}^{2g+1}_{0}\cup\mathcal{JD}^{2g+1}_{4}$$ and restricting to the hyperelliptic curves we have:
$$\Is^{2g+1}_{\underline{2}}\cap\mathcal{JH}=\mathcal{JHD}^{2g+1}_{0}=\mathcal{JHD}^{2g+1}_{4}.$$
For even dimension we have:
$$\Is^{2g}_{\underline{2}}\cap\cJ=\mathcal{JD}^{2g}_2, \text{ and }\ 
\Is^{2g}_{\underline{2}}\cap\mathcal{JH}=\mathcal{JHD}^{2g}_2$$

\end{theorem}
\begin{proof}
The idea of the proof is that the inclusion $\supset$ comes from the Prym construction and $\subset$ comes from the generalised Torelli Theorem.

We will prove the odd dimension case in detail.
Let $C\in\mathcal{D}^{2g+1}_{0}\cup\mathcal{D}^{2g+1}_{4}$.
Let $f$ be the quotient map and $C'$ be the quotient curve.
If $C\in\mathcal{D}^{2g+1}_{4}$ then $C'$ is of genus $g$ and $f^*$ is injective, so $JC'$ is embedded in $JC$ with the induced polarisation of type $\underline{2}$, whereas for $C\in\mathcal{D}^{2g+1}_{0}$ we have that $f$ is \'etale defined by the two torsion point, say $\eta$, and $C'$ is of genus $g+1$, so $JC'/\eta$ is embedded in $JC$ with the induced polarisation of type $(1,2,\ldots,2)$. The complementary polarisation type is $\underline{2}$, so in both cases $JC\in\Is^{2g+1}_{\underline{2}}$.

Now, let $JC\in\Is^{2g+1}_{\underline{2}}$. 
Let $M$ be a subvariety of $JC$ with the
induced polarisation of type $\underline{2}$ and let $N$ be the complementary subvariety to $M$. Then $\Nm_M^2=2\Nm_M$ and
therefore $i_M=(1-\Nm_M)$ is an involution of $JC$, because
$$(1-\Nm_M)^2=1-2\Nm_M+\Nm_M^2=1.$$
As $\Nm_N=2-\Nm_M$, we get that $i_N=(1-\Nm_N)=-i_M$ is another involution.

By Theorem \ref{gentorthm} one of them comes from the involution on $C$.
Assume that $i_M$ is the extension of the involution $\iota_M\in\Aut(C)$.
Let $C'=C/\iota_M$ and let $f\colon C\To C'$ be the quotient map.
Then for $c\in C$ we have 
$$f^*f(c)=f^*([c,\iota_M(c)])=c+\iota_M(c)=\Nm_N(c).$$
As $f^*:JC'\To JC$ is finite, it means that the genus of $C'$ equals $$g(C')=\dim(JC')=\dim(N)=g+1.$$
From the Hurwitz formula we get that $2(2g+1)-2=2(2(g+1)-2)+b$, so $b=0$
and therefore $f$ is an \'etale double cover, so $JC\in\mathcal{JD}^{2g+1}_{0}$
Analogously, if $i_N$ is the extension of the involution on $C$ then $f$ is branched in 4 points, so $JC\in\mathcal{JD}^{2g+1}_{4}$.

If $C$ is a hyperelliptic curve, then both involutions $i_M$ and $i_N$ come from involutions on $C$, so
$\Is^{2g+1}_{\underline{2}}\cap\mathcal{JH}\subset\mathcal{JHD}^{2g+1}_{0}$ and 
$\Is^{2g+1}_{\underline{2}}\cap\mathcal{JH}\subset\mathcal{JHD}^{2g+1}_{4}$
 which gives the second part of the theorem.

In the even dimensional case, both $M$ and $N$ are of dimension $g$ and with the induced polarisation of type $\underline{2}$ and the Hurwitz formula gives $b=2$.

In particular when $C$ is not hyperelliptic and
$JC\in\Is^{2g}_{\underline{2}}$, then for a pair of complementary
subvarieties $(M,N)$ of type $\underline{2}$ exactly one of them is
the Jacobian $JC'$ of genus $g$ curve such that $C$ is a double cover
of $C'$ and the other is the Prym variety of the double cover.
If $C$ is hyperelliptic, then both subvarieties are Jacobians and
Pryms for each other.
\end{proof}

The idea of this proof leads to an interesting observation.
\begin{Cor}\label{corIsJempty}
Let $D=(1,\ldots,2)$ be a $g$-tuple with a positive number
of $1$'s and $2$'s. Then
\begin{enumerate}
\item $\Is^{2g}_{D}\cap\cJ=\emptyset$,
\item $\Is^{2g+1}_{D}\cap\cJ=\emptyset$.
\end{enumerate}
\end{Cor}
\begin{proof}
If either of those was non-empty, we would find a pair of
complementary subvarieties $(M,N)$ and an involution on $C$ which
induces one of the involutions $(1-\Nm_M)$ or $(1-\Nm_N)$. Taking the
quotient curve $C'$, with quotient map $f$, we would find that
$M=f^*(JC')$ or $N=f^*(JC')$.

In the first case both subvarieties are of dimension $g$,
so the Hurwitz formula tells us that $f$ has to be a double cover
branched in $2$ points, which means that the induced polarisation on
$f^*JC'$ is twice a principal polarisation, a contradiction.

In the second case the Hurwitz formula states that
$2(2g+1)-2 =2(2g(C')-2)+b$, where $b$ is the number of branch points, and gives two possibilities. If $g(C')=g+1$,
then $b=0$ and we have an \'etale double cover and from the Prym construction (\cite[Thm 12.3.3]{LB}) we get a contradiction. If $g(C')=g$, then $b=4$ and again we get a
contradiction because the induced polarisation on $N$ has to be of type $\underline{2}$ (\cite[Prop 11.4.3]{LB}).
\end{proof}

There is one more result related to $\Is^3_2$.
\begin{prop}
There is a $1$ to $1$ correspondence between the set of smooth genus~$3$
hyperelliptic curves (up to translation) on a general abelian surface
$A$ and the set of degree~$2$ polarised isogenies $A\To B$,
where $B$ is the Jacobian of a smooth genus~$2$ curve. In particular,
there are exactly three hyperelliptic curves in the linear system of a
$(1,2)$ polarising line bundle on a (very) general abelian surface.
\end{prop}

\begin{proof}
Let $(JC',\Theta)$ be the Jacobian of a smooth genus~$2$ curve. Let
$\rho\colon A\To JC'$ be a degree~$2$ isogeny. Then
$\rho^{-1}(\Theta)$ is a genus~$3$ hyperelliptic curve on $A$, which is an
\'etale double cover of $C'$.

Conversely, let $C$ be a hyperelliptic genus~$3$ curve on $A$. Then
$\cO(C)$ is a $(1,2)$ polarising line bundle. From the universal
property of Jacobians there exists a surjective
map $f\colon JC\To A$. 
By \cite[Prop 4.3]{BL2}, $\hat{f}$ is an embedding of $\hat{A}$
with restricted polarisation of type $(1,2)$.  Therefore
$JC\in\Is^3_2$.  As $C$ is also hyperelliptic,
Proposition~\ref{propIs32} tells us that there exists an \'etale
double cover $C\To C'$. It is defined by a $2$-torsion point, say
$\eta$, and there is an embedding of $JC'/\langle\eta\rangle$ to $JC$ with the
restricted polarisation of type $(1,2)$. As $A$ is general, we have
$\hat{A}=JC'/\langle\eta\rangle$ and by dualising the quotient map, we obtain
a degree~$2$ polarised isogeny $A\To JC'$.

The last part follows from the fact that there are exactly three
non-zero $2$-torsion points in the kernel $K(\cO(C))$. A (very)
general surface means one for which the resulting principally
polarised abelian surface is the Jacobian of a smooth curve.
\end{proof}

\subsection{Irreducibility}\label{secirr}
The aim of this section is to show that $\Is^g_{D}$ is irreducible.
This will be an indication that the choice of definition is a good one.

In the proof of this fact we will use condition~(4) of
Proposition~\ref{cmplandroots}, so we define
\begin{definition}
Let $M,N,A$ be polarised abelian varieties. An \emph{allowed isogeny} is a
polarised isogeny $\rho\colon M\times N\To A$, such that its
kernel has $\{0\}$ intersection with $M\times\{0\}$ and $\{0\}\times
N$.
\end{definition}

\begin{definition}
Let $(M,H_M),(N,H_N)$ be polarised abelian varieties of type
$D$ and $\tilde{D}$. A subgroup $K\subset M\times N$
is called an \emph{allowed isotropic subgroup} if it is a maximal
 isotropic subgroup of $K(H_M\boxtimes
H_N)$, with respect to $e^{H_M\boxtimes H_N}$, such that $K\cap
K(H_M)=K\cap K(H_N)=\{0\}$. Note that every maximal isotropic subgroup of 
$K(H_M\boxtimes H_N)$ has order $(d_1\cdot\ldots\cdot d_k)^2$.
\end{definition}
Let us recall
\begin{prop}\cite[Cor 6.3.5]{LB}\label{prop635}
For an isogeny $\rho\colon Y\To X$ and $L\in\Pic(Y)$ the
following statements are equivalent
\begin{enumerate}
\item $L=\rho^*(L')$ for some $L'\in\Pic(X)$. 
\item $\ker(\rho)$ is an isotropic subgroup of $K(L)$ with respect to $e^L$. 
\end{enumerate}
\end{prop}

This leads to an obvious corollary.
\begin{Cor}
Let $A$ be a principally polarised abelian variety. Let $M,N$ be
polarised abelian varieties of type $D$ and
$\tilde{D}$. Then
\begin{enumerate}
\item If $\rho\colon M\times N\To A$ is an allowed isogeny
  then $\ker(\rho)$ is an allowed subgroup of $M\times N$.
\item If $K$ is an allowed subgroup of $M\times N$, then $(M\times
  N)/K$ is a principally polarised abelian variety and the quotient
  map $\rho\colon M\times N\To (M\times N)/K$ is an
  allowed isogeny.
\end{enumerate}
\end{Cor}

Let us state the main result of this section.
\begin{prop}\label{irrofis}
Let $k,g$ be integers such that
$0<k\leq \frac{g}{2}$, and $D=(d_1,\ldots, d_k)$ 
be a polarisation type.
Then $\Is^g_{D}$ is an irreducible subvariety of
  $\cA_g$ of dimension $(^{k+1}_{\ 2})+(^{g-k+1}_{\ \ \ 2\ })$ and codimension $k(g-k)$. 
\end{prop}
\begin{proof}
Proposition~\ref{cmplandroots} tells us that $A$ belongs to
$\Is^g_{D}$ if and only if there exists an allowed isogeny
to $A$.  Therefore the idea of the proof is to show that there exists
one map from $\SH_k\times\SH_{g-k}$ which covers all
possible allowed isogenies and so $\Is^g_{D}$ is the image
of an irreducible variety.

The sketch of the proof is as follows. Take polarised abelian varieties $(M,
H_M)$ and $(N,H_N)$ of types $D$ and $\tilde{D}$
respectively. Take their product with product polarisation $(M\times
N,H_M\boxtimes H_N)$.  By \cite[Lem 3.1.4]{LB}, we have
$K(H_M)\cong K(H_N)$ and $K(H_M\boxtimes H_N)$, of order $\prod d_i^4$,
is
a symplectic $\ZZ$-module with the non-degenerate symplectic
form $e^{H_M\boxtimes H_N}$. Therefore there exists an allowed
isotropic subgroup $G\subset K(H_M\boxtimes H_N)$ and by Proposition~\ref{maxisosub2} all such are equivalent under the
action of the symplectic group. Hence there exists an allowed isogeny
$\rho\colon M\times N\To (M\times N)/G$ and so
$\Is^g_{D}$ is non-empty.  Moreover, the action of the
symplectic group on $K(H_M\boxtimes H_N)$ is induced by the
symplectic action on $\SH_g$, which gives us irreducibility.

To make this more precise, we need to recall that a period matrix of
an abelian variety is a choice of symplectic basis of a lattice in its
universal cover.

Let $l\leq k$ be the number of integers bigger than $1$ in
$D$. By the elementary divisor theorem, let
$B^M=\{\lambda_1^M,\ldots,\lambda_l^M,\mu_1^M,\ldots,\mu_l^M\}$ and $B^N=\{\lambda_1^N,\ldots,\lambda_l^N,\mu_1^N,\ldots,\mu_l^N\}$ be
symplectic bases of $K(H_M)$ and $K(H_N)$. Then $B=\{B^M,B^N\}$ is a
symplectic basis of $K(H_M\boxtimes H_N)$.
Let $K_B$ be given by the image (that is, the group generated by the columns) of the matrix 
$$K=\left[\begin{array}{cc}
    \id_l&0\\ 0&\id_l\\ \id_l&0\\ 0&-\id_l
\end{array}\right].$$
Then $K_B$ is an allowed isotropic subgroup. Moreover, if we change
bases using the symplectic action, then $\image(K)$ will always define an
allowed isotropic subgroup and by Proposition~\ref{maxisosub2}, every allowed isotropic subgroup arises in this way.

When we take the universal cover $V$ of $M\times N$, in order to write
the period matrix we need to choose a symplectic basis of $V$. The
obvious choice is to enlarge the symplectic basis $B$ to a symplectic
basis $\overline{B}$. We need to enlarge the matrix $K$ by zero-blocks
to a matrix $\overline{K}$, such that its image is equal to $K_B$.

From this discussion, we have found
$$
\Lambda= \left<\begin{array}{cccc}Z(M)&0&\diag(D)&0\\
0&Z(N)&0&\diag(\tilde{D})
\end{array}\right>,\  M\times N=\CC^{g}/\Lambda,
$$
and a matrix $\overline{K}$, such that $\image(\overline{K})$ is an
allowed isotropic subgroup of $M\times N$.

The data defining $\overline{K}$ is discrete, so
$\image(\overline{K})$ will be allowed isotropic for any matrices
$Z_k\in\SH_k,Z_{g-k}\in\SH_{g-k}$.  Moreover, the
symplectic action on $\SH_k\times\SH_{g-k}$ gives
all possible period matrices hence all possible symplectic bases and
therefore all possible allowed isotropic subgroups.

Thus we have proved that there exists a global map
$$
\Psi\colon\SH_{k}\times\SH_{g-k}\ni Z_k\times
Z_{g-k}\mapsto (A_{Z_k}\times A_{Z_{g-k}})/\image(\overline{K})\in
\cA_g,
$$ 
which covers all possible allowed isogenies; that is, for any allowed isogeny $M\times N\To A$, there exist period matrices
$Z(M)$ and $Z(N)$ such that $\Psi(Z(M)\times Z(N))=A$.

From the construction, it is obvious that $\Is^g_{D}$ is the
image of the above map and as the domain is irreducible, it follows that
$\Is^g_{D}$ is an irreducible variety.
\end{proof}
\begin{remark}
Proposition \ref{irrofis} is stated as a fact in \cite[(9.2)]{Deb} and proved using irreducibility of some moduli space. Both constructions are similar, but the proof presented in this paper is explicit.
\end{remark}

There is a generalisation of Humbert surfaces to the moduli of
non-principally polarised abelian surfaces. However, in that case, the
generalised Humbert surface is no longer irreducible. For details, see
\cite{V}.

One can also generalise further Proposition \ref{irrofis} to 
non-principally polarised abelian varieties. 
Let $D$ be a polarisation type of an abelian $g$-fold.
The idea is to define for
any polarisation types $D_1\in\ZZ^k$, $D_2\in\ZZ^{g-k}$ the locus
$\Is^{g,D}_{D_1,D_2}$ of $D$-polarised abelian $g$-folds
which have a pair of complementary subvarieties of types
$D_1$ and $D_2$. The obvious question is
whether $\Is^{g,D}_{D_1,D_2}$ is non-empty. Using
Proposition~\ref{prop635} one can translate the question into one
about the
existence of isotropic subgroups analogous to the allowed ones. The
proof of Proposition~\ref{irrofis} can be easily generalised, but one must
have in mind that the number of irreducible components of
$\Is^{g,D}_{D_1,D_2}$ will be equal to the number
of orbits of such isotropic subgroups. To sum up, the problem can be
solved if one can deal with the combinatorics related to special
isotropic subgroups in finite symplectic groups. Certainly this is
possible in many cases, such as $(1,p)$-polarised surfaces (see
\cite{V}).

\section{Equations in the Siegel space}\label{sec4}
As in the Humbert surface case, we would like to find equations for a
locus in $\SH_g$ which maps to $\Is^g_{D}$ in $\cA_g$. Ideally, we would like to find the equations of the whole preimage of $\Is^g_{D}$ which would involve understanding the action of $\Sp(2g,\ZZ)$ on $\SH_g$ and finding good symplectic invariants. 

We start by proving an obvious, yet important lemma.
\begin{lemma}\label{lsublattice}
$(A,H)=(\CC^g/\Lambda,H)$ is non-simple if and only if there exists a $g>k>0$ dimensional complex subspace $V$ such that $\Lambda\cap V$ is symplectic of rank $2k$.
\end{lemma}
\begin{proof}
If $A$ is non-simple, then there exists a subvariety $B$ of dimension, say $k$. Taking a universal cover of $A$, we get a lattice $\Lambda$ and a hermitian form $H$. The preimage of $B$ is a vector subspace, say $V$ and $\Lambda\cap V$ is of rank $2k$. The induced polarisation on $B$ is given by restriction of $H$, so $\Lambda\cap V$ has to be symplectic sublattice. 

If $V$ is a complex subspace then we define $B=(V,\Lambda\cap V)$. By assumption, $H|_V$ is a polarisation, so $B$ is an abelian subvariety, so $A$ is non-simple.
\end{proof}

\subsection{Equations of non-simple abelian varieties}
The idea of constructing equations is to choose a sublattice that is symplectic of type $D=(d_1,\ldots,d_k)$ and, by applying equations, force it to lie in a complex subspace. To fix notation, let $e_i$ denote the basis of $\CC^g$ and $f_i=Z_A(i)$ the column vectors of some Siegel matrix $Z_A$. Then $\ ^te_i\Ima(Z_A)^{-1}f_j=\delta_{i,j}$ gives a standard matrix of the symplectic form.
Choose $g_i=d_ie_i+e_{g+1-i}, i=1,\ldots,k$. Obviously $g_i$'s are primitive, linearly independent and well defined since $k\leq\frac{g}{2}$. Moreover, $\{g_i,f_i,\ i=1,\ldots,k\}$ generates a symplectic sublattice of type $D$.
The following equations will force $f_i$'s to lie in the complex subspace generated by $g_i$'s and by Lemma \ref{lsublattice} will give the desired outcome. From now on, we start to abuse notation by writing 0 for the block matrix of the correct dimension.

\begin{theorem}\label{thmIsgd}
Let $0<k\leq\frac{g}{2}$ and let $D=(d_1,\ldots,d_k)$ be a type of polarisation. 

Let $Z_A=[z_{ij}]\in \SH_g$ satisfy 
\begin{align*}
&z_{ij}=d_iz_{(g+1-i)j},&\quad &i=1,\ldots,k, j=1,\ldots,k&\\ 
&z_{ij}=0,&\quad &i=k+1,\ldots,g-k, j=1,\ldots,k&
\end{align*}

Let $\Lambda_A=\left<Z_A\ \id_g\right>$ and $A=\CC^g/\Lambda_A$.
Then $A\in\Is^g_D$, i.e.
there exists an abelian subvariety given by $$
Z_M=\left[\begin{array}{cccc}
  z_{11}&z_{12}&\dots&z_{1k}\\
  z_{12}&z_{22}&\dots&z_{2k}\\
  \vdots&\vdots&\ddots&\vdots\\
  z_{1k}&z_{2k}&\dots&z_{kk}\\
\end{array}\right],\ \Lambda_M=<Z_M,D>,\ 
M=\CC^k/\Lambda_M,
$$
and an embedding
$$
\iota_M\colon
M\To A
$$
$$ 
(x_1,\ldots,x_k)+\Lambda_M\Mapsto
(x_1,\ldots,x_k,0,\ldots,0,\frac{x_1}{d_1},\ldots,\frac{x_k}{d_k})+\Lambda_A
$$

such that the restricted polarisation is of type $D$.
\end{theorem}
\begin{proof}
To shorten notation we will write $Z_M(i)$ for the $i$-th column vector of the matrix $Z_M$ and $z'$ for $\Ima(z)$. By $\frac{1}{D}$ we will denote the matrix $\diag(\frac{1}{d_1},\ldots,\frac{1}{d_k})$.

Obviously $\iota_M$ is a well defined embedding because the images of generators of $\Lambda_M$ are given by
$\iota_M(d_ie_i)=d_ie_i+e_{g+1-i}$ and
$\iota_M(Z_M(i))=Z_A(i)$
and thus they are primitive vectors in $\Lambda_A$.

It remains to compute the restricted polarisation using analytic representations written in block matrices. We have
\begin{align*}
\left[\begin{array}{cccc}
\id_k&0&\frac{1}{D}\end{array}\right]&
(\Ima Z_A)^{-1}
\left[\begin{array}{c}
\id_k\\0\\\frac{1}{D}\end{array}\right]\Ima Z_M
=\\
\left[\begin{array}{cccc}
\id_k&0&\frac{1}{D}\end{array}\right]&
\left[\begin{array}{cccc}
\Lambda_A'(1)&\Lambda_A'(2)&\ldots&\Lambda_A'(g)
 \end{array}\right]^{-1}
\left[\begin{array}{cccc}
\Lambda_A'(1)&\Lambda_A'(2)&\ldots&\Lambda_A'(k)
 \end{array}\right]=\\
\left[\begin{array}{cccc}
\id_k&0&\frac{1}{D}\end{array}\right]&
\left[\begin{array}{c}
\id_k\\0\\0\end{array}\right]=\id_k,
\end{align*}
so the induced polarisation is of type $D$.
\end{proof}
\begin{remark}
Let $\tilde{D}=\diag(1,\ldots,1,d_1,\ldots,d_k)$ be a $(g-k)$-tuple and 
$\tilde{Z}_M=\left[\begin{array}{cc}
  0&0\\
  0&Z_M\\
  \end{array}\right]$
be the zero extension of $Z_M$ to a $(g-k)\times(g-k)$ symmetric matrix.
Let
$$X=\left[\begin{array}{cccc}
  z_{k+1k+1}&z_{k+1k+2}&\dots&z_{k+1g}\\
  z_{k+1k+2}&z_{k+2k+2}&\dots&z_{k+2g}\\
  \vdots&\vdots&\ddots&\vdots\\
  z_{k+1g}&z_{k+2g}&\dots&z_{gg}\\
\end{array}\right].$$
Then
$$Z_N=\tilde{D}X\tilde{D}-\tilde{Z}_M$$ is a symmetric matrix and the complementary abelian subvariety is given by
$$ \Lambda_N=<Z_N,\ \tilde{D}>,\ 
N=\CC^{g-k}/\Lambda_N,$$
and 
$$
\iota_N\colon
N\To A
$$
$$ 
(y_1,\ldots,y_{g-k})+\Lambda_N\Mapsto
(0,\ldots,0,\frac{y_1}{\tilde{d}_1},\ldots,\frac{y_{g-k}}{\tilde{d}_{g-k}})+\Lambda_A
$$
is a well defined embedding.

To shorten computations, we will describe only important steps in block matrices. Then
$$\iota_N(\tilde{D})=\left[\begin{array}{c}
  0\\
  \id_{g-k}\\
  \end{array}\right] \ \text{ and }\ \iota_N(Z_N)=\left[\begin{array}{c}
  0\\
  X\tilde{D}-\tilde{D}^{-1}\tilde{Z}_M\\
  \end{array}\right]$$
is the $g\times (g-k)$ block matrix of the last $(g-k)$ columns of $Z_A$ with the last $k$ columns of $Z_A$ multiplied by $D$ from the right and having subtracted the first $k$ columns. Therefore, images of generators of $\Lambda_N$ are primitive so $\iota_N$ is an embedding. Checking that the restricted polarisation is of type $\tilde{D}$ is completely analogous.
\end{remark}

Using Propositions~\ref{cmplandroots} and~\ref{irrofis}, we can summarise the discussion into the following theorem.
\begin{theorem}\label{Is1p}
Let $(A,H)$ be a principally polarised abelian variety and suppose 
$D$ and $\tilde{D}$ are possible polarisation types of complementary abelian subvarieties.
The following conditions are equivalent:
\begin{enumerate}
\item there exists an abelian subvariety $M\subset A$ such that
  $H|_M$ is of type $D$, i.e.\ $A\in\Is^g_{D}$;
\item there exists a pair $(M,N)$ of complementary abelian subvarieties in
  $A$ of types $D$ and $\tilde{D}$;

\item $(A,H)$ is isomorphic to an abelian variety defined by the lattice
$\Lambda_A=\left<Z_A\ \id_g\right>$,
with $Z_A=[z_{ij}]\in \SH_g$
satisfying the linear equations
$$\begin{cases}
z_{ij}=d_iz_{(g+1-i)j},\quad i=1,\ldots,k,\quad\quad\quad\quad j=1,\ldots,k\\ 
z_{ij}=0,\quad \quad\quad\quad\quad i=k+1,\ldots,g-k,\ \ j=1,\ldots,k
\end{cases}$$

\end{enumerate}
\end{theorem}
\begin{proof}
$(1)\Leftrightarrow (2)$ is the content of Proposition~\ref{cmplandroots},
so we only need to prove that $(1)\Leftrightarrow (3)$.  Theorem~\ref{thmIsgd}
tells us that the set of abelian varieties with a period matrix
satisfying~(3) is a subset of $\Is^g_D$. Moreover, both of them
are closed irreducible subvarieties of $\cA_g$ of codimension~$k(g-k)$,
which means that they are equal. In other words, the locus
\begin{align*}
\{Z=[z_{ij}]\in\SH_g\colon
&z_{ij}=d_iz_{(g+1-i)j},&\quad &i=1,\ldots,k,\ j=1,\ldots,k&\\ 
&z_{ij}=0,&\quad &i=k+1,\ldots,g-k,\ j=1,\ldots,k\}&
\end{align*} 
is one of the irreducible components of the preimage of $\Is^g_D$ in
$\SH_g$.
\end{proof}

\end{document}